\documentclass[a4paper, 11pt]{article}

\usepackage[utf8]{inputenc}
\usepackage[T1]{fontenc}
\usepackage{fullpage}

\usepackage{comment}
\usepackage{amsmath, amsthm, amssymb}
\usepackage{graphicx}
\usepackage{enumerate}
\usepackage{authblk}
\usepackage{thmtools}
\usepackage{thm-restate}
\usepackage[colorlinks=true, citecolor=red]{hyperref}
\usepackage{float}

\DeclareMathOperator{\diam}{\mathrm{diam}}

\usepackage[linesnumbered,ruled,onelanguage,vlined]{algorithm2e}

\usepackage{tikz}
\renewenvironment{abstract}
{\small\vspace{-1em}
\begin{center}
\bfseries\abstractname\vspace{-.5em}\vspace{0pt}
\end{center}
\list{}{
\setlength{\leftmargin}{0.6in}\setlength{\rightmargin}{\leftmargin}}\item\relax}
{\endlist}
\usepackage{comment}
\declaretheorem[name=Theorem]{theorem}
\declaretheorem[name=Lemma, sibling=theorem]{lemma}

\declaretheorem[name=Corollary, sibling=theorem]{corollary}
\declaretheorem[name=Conjecture, sibling=theorem]{conjecture}

\declaretheorem[name=Claim, sibling=theorem]{claim}

\declaretheorem[name=Observation, style=remark, sibling=theorem]{observation}
\declaretheorem[name=Question, style=remark, sibling=theorem]{question}
\def\cqedsymbol{\ifmmode$\lrcorner$\else{\unskip\nobreak\hfil
\penalty50\hskip1em\null\nobreak\hfil$\lrcorner$
arfillskip=0pt\finalhyphendemerits=0\endgraf}\fi}

\renewcommand{\bar}{\overline}
     \def\C{\mathcal{C}}        
    \def\P{\mathcal{P}} \def\R{\mathcal{R}}

\usetikzlibrary{calc}
\let\le\leqslant
\let\ge\geqslant

\title{Extremal Independent Set Reconfiguration\thanks{This work was supported by ANR project GrR (ANR-18-CE40-0032)}}

\author[1]{Nicolas Bousquet}
\author[1,2]{Bastien Durain}
\author[1]{Théo Pierron}
\author[2]{Stéphan Thomassé}

\affil[1]{Univ. Lyon, Université Lyon 1, LIRIS UMR CNRS 5205, F-69621, Lyon, France}
\affil[2]{ENS Lyon, Département informatique, Lyon, France}
\date{}

\begin{document}

\maketitle
 
 \begin{abstract}
     The independent set reconfiguration problem asks whether one can transform one given independent set of a graph into another, by changing vertices one by one in such a way the intermediate sets remain independent. 
     Extremal problems on independent sets are widely studied: for example, it is well known that an $n$-vertex graph has at most $3^{n/3}$ maximum independent sets (and this is tight). 
     This paper investigates the asymptotic behavior of maximum possible length of a shortest reconfiguration sequence for independent sets of size $k$ among all $n$-vertex graphs. 
     
     We give a tight bound for $k=2$. We also provide a subquadratic upper bound (using the hypergraph removal lemma) as well as an almost tight construction for $k=3$. We generalize our results for larger values of $k$ by proving an $n^{2\lfloor k/3 \rfloor}$ lower bound.
 \end{abstract}

\section{Introduction}

Many questions can be formalized as follows:
given the description of a system state and the description of a state we would ``prefer'' the system to be in, is it possible to transform the system
from its current state into the more desired one without ``breaking'' the system in the process? And if yes, how many steps are needed? Such problems naturally arise in the fields of mathematical puzzles, operational research, computational geometry, bioinformatics,  and quantum computing for instance. These questions received a substantial amount of attention under the so-called \emph{combinatorial reconfiguration framework} in the last few years from both structural and algorithmic point of views.
We refer the reader to the surveys~\cite{H13,DBLP:journals/algorithms/Nishimura18,BousquetMNS22+} for more background on combinatorial reconfiguration.

Given a reconfiguration problem, one can naturally define the \emph{(re)configuration graph} where the vertices correspond to solutions and there is an edge between two vertices if one can transform one into the other in one step. Structural properties of configuration graphs have been studied under various names in different fields, for instance by looking for its connectivity (irreducibility of Markov chains) or hamiltonian paths (Gray codes for hypercubes).

\paragraph*{Independent set reconfiguration.} 
Given a simple undirected graph $G$, a set of vertices $S \subseteq V(G)$ is an \emph{independent set} if the vertices
of $S$ are pairwise non-adjacent. Finding an independent set of maximum cardinality, i.e., the {\sc Independent Set} problem,
is a fundamental problem in algorithmic graph theory and is known to be not only NP-hard, but also W[1]-hard and
not approximable within $O(n^{1-\varepsilon})$, for any $\varepsilon > 0$, unless P $=$ NP~\cite{DBLP:journals/toc/Zuckerman07}.

We view an independent set as a collection of tokens placed on the vertices of a graph such that no two tokens are adjacent.
This gives rise to two natural adjacency relations between independent sets, also called \emph{reconfiguration steps}.
In the {\sc Token Jumping} (TJ) problem, introduced by Kami\'{n}ski et al.~\cite{KMM12}, a single reconfiguration step consists of first removing a token on some vertex $u$ and then immediately adding it back on any other vertex $v$, as long as no two tokens become adjacent. The token is said to \emph{jump} from vertex $u$ to vertex $v$.
In the {\sc Token Sliding} (TS) problem, introduced by Hearn and Demaine~\cite{DBLP:journals/tcs/HearnD05}, two independent sets are adjacent if one can be obtained from the other by a token jump from vertex $u$ to vertex $v$ with the additional requirement of $uv$ being an edge of the graph. The token is then said to \emph{slide} from vertex $u$ to vertex $v$ along the edge $uv$. Note that, in both the TJ and TS problems, the size of independent sets is fixed.
Generally speaking, in the {\sc Token Jumping} and {\sc Token Sliding} problems,
we are given a graph $G$ and two independent sets $I_s$ and $I_t$ of $G$. The goal is to determine
whether there exists a sequence of reconfiguration steps (called a \emph{reconfiguration sequence}) that
transforms $I_s$ into $I_t$ (where the reconfiguration step depends on the problem).

We can reformulate the problem with the configuration graph. Given a graph $G$ we can define the \emph{configuration graph} $\R_k(G)$ as the graph whose vertices correspond to independent sets of size $k$ and where we put an edge between $I$ and $J$ if one can transform $I$ into $J$ in one step (under the token jumping variant).  There exists a reconfiguration sequence from $I$ to $J$ if and only if $I$ and $J$ belong to the same connected component of $\R_k(G)$.

Both problems have been extensively studied, albeit under
different names. They are PSPACE-complete, even restricted to bounded bandwidth (and hence pathwidth) graphs~\cite{WROCHNA14} and planar graphs~\cite{DBLP:journals/tcs/HearnD05}. Their complexity is also known (respectively PSPACE and NP) on bipartite graphs~\cite{DBLP:journals/talg/LokshtanovM19} and several polynomial algorithms exist in simpler classes such as trees~\cite{demaine_linear-time_2015} and interval graphs~\cite{DBLP:conf/wg/BonamyB17}. 

All along the paper we mainly focus on the Token Jumping model but all our lower bounds also hold for the Token Sliding version.

\paragraph{Diameter of the configuration graph and the $(6,3)$-problem.}
In many cases, the diameter of the configuration graph, even if connected, is not polynomial (and that is one of the reasons why most of the reconfiguration problems do not belong to NP). An important line of research has focused on finding conditions that ensure that the configuration. But the asymptotic behavior of maximum possible length of a shortest reconfiguration sequence has not been really studied. 
 The problem of determining "which graphs on $n$ vertices have the largest amount of independent sets?" has received considerable attention. On the contrary, the question "which graphs on $n$ vertices have a configuration graph of independent sets has the largest diameter?" has not, as far as we know, received any attention. 

The $3n$-vertex graph with the largest number of maximum independent sets is a disjoint collection of triangles which admits $3^{n}$ independent sets. In that case, one can easily remark that we can easily transform any maximum independent set into any other in $O(n)$ steps by replacing a vertex of a triangle by another (which can be done without conflict since the triangles are independent). 
So a graph whose configuration graph of independent sets has maximum diameter must have a completely different behavior. In this paper, we consider the following questions: what is the largest possible diameter of (a connected component of) the configuration graph amongst all the graphs of size $n$? What if we fix the size $k$ of the independent set we want to consider?

Let $k,n$ be two integers. Let us denote by $D(n,k)$ the maximum diameter, amongst all the graph $G$ on $n$ vertices, of a connected component of $\R_k(G)$. The goal of this paper mainly focuses on finding lower and upper bounds on $D(n,k)$.
There is a natural upper bound for $D(n,k)$ which is the maximum number ${n \choose k}$ of subsets of vertices size $k$.  We will prove in Section~\ref{sec:upper} that this bound cannot be reached and that the $D(n,k)$ is actually at most $O(n^{k-1})$. More precisely, we will prove that $D(n,k)\leqslant {n \choose k-1}$. 

We can easily prove that the order of magnitude of this bound is tight since, for $k=2$, the following holds as we will prove in Section~\ref{sec:lower}:

\begin{restatable}[]{theorem}{twoIS}
\label{thm:2IS}
 $D(n,2) = n-2$, and the complement of the $n$-vertex path is the unique tight example.
\end{restatable}

One can naturally wonder if this bound is still tight for larger values of $k$. The answer is negative since we can prove that this upper bound can actually be very slightly improved for every $k \ge 3$. Namely, we will prove that:

\begin{restatable}[]{theorem}{littleo}
\label{thm:bornesup_gen}
For $k\geqslant 3$, we have $D(n,k)=o(n^{k-1})$. 
\end{restatable}

The proof of Theorem~\ref{thm:bornesup_gen} is inspired from the upper bound proof of the $(6,3)$-problem and is based on an application of the hypergraph removal lemma. 
A hypergraph $\mathcal{H}$ is \emph{$(s,t)$-free} if no set of $s$ vertices of $\mathcal{H}$ contains at least $t$ hyperedges. The $(6,3)$-problem (or Ruzsa–Szemerédi problem) asks for the maximum number of hyperedges in a $(6,3)$-free $n$-vertex $3$-uniform hypergraph. The so-called $(6,3)$-theorem of Ruzsa-Szemerédi~\cite{ruzsa1978triple} ensures this value is $o(n^2)$.

This gain ($o(n^{k-1})$ versus $O(n^{k-1})$) might appear marginal but we can prove that, again, it cannot be widely improved. Namely we prove that the following holds:

\begin{restatable}[]{theorem}{threeIS}
\label{thm:3IS}
\[ D(n,3)= \Omega(n^2/e^{O(\sqrt{\log n})}). \]
\end{restatable}

The value $n/e^{O(\sqrt{\log n})}$ corresponds to the largest known asymptotic size for a subset of $[1,n]$ without arithmetic progressions of length $3$~\cite{behrend1946sets}. Any improvement of this bound would also imply an improvement of the bound of Theorem~\ref{thm:3IS}. Note that the best bound for the $(6,3)$-problem also has this order of magnitude~\cite{ruzsa1978triple}.

The $3$-reconfiguration problem is actually very close to the $(6,3)$-problem. Indeed, if we consider a shortest path in the $3$-configuration graph of $G$ and only consider even (resp. odd) vertices of that path, then we have a set of hyperedges of size $3$. And one can easily check that this set of hyperedges satisfy the $(6,3)$-property. So our result implies in particular that, given a set of size $n$, we can find two sets $X_1,X_2$ of $n^2/e^{O(\sqrt{\log n})}$ $3$-hyperedges such that both of them are $(6,3)$-free but whose union is "path-like", meaning that for every hyperedge (but at most two which are the endpoints of the path) there are two others hyperedges that intersect it on two vertices. 

The idea of the proof of Theorem~\ref{thm:3IS} consists in starting from a clique. We will then remove edges to create almost linearly many paths in the configuration graph of linear length. The involved part of the proof consists in showing that these paths remain independent of each other (i.e. there is no edge between them in the configuration graph) using a set $S$ of integers with no arithmetic progression of size $3$. We finally use a last trick to glue these paths together in order to obtain the claimed diameter. Note that the classical construction giving $n/e^{O(\sqrt{\log n})}$ hyperedges~\cite{ruzsa1978triple} for the $(6,3)$-problem cannot be easily used in our construction since the construction is tripartite and then hard to reconnect into a configuration graph.

We were not able to prove that the lower bounds and the upper bounds almost match for larger values of $k$. In particular, it is open to determine if the $4$-configuration graph can have super-quadratic diameter (while the upper bound is $o(n^3)$). We conjecture that the following holds:

\begin{conjecture}
\[D(n,4)=n^{3-o(1)}.\]
\end{conjecture}

A first step to prove super-quadratic diameter is to ensure that there exists a graph with a lot of copies of $K_4$ such that no two of them intersect on a triangle. This was recently shown to be true for any value of $k$. Namely, Gower and Janzer proved in~\cite{gowers2021generalizations} that, for every $k$ and every $n$, there exists an $n$-vertex graph with
$n^{k-1-o(1)}$ copies of $K_k$ such that every $K_{k-1}$ is contained in at most one $K_k$. This result might suggest that $D(n,k)=n^{k-1-o(1)}$.

Our construction for $k=3$ has to be drastically modified in order to work. Indeed, our construction is heavily based on the fact that we can find a graph with an almost linear number of linear paths in its $3$-configuration graph. To get a super-quadratic bound, we need to either increase the number of paths or their lengths. We failed trying both options.

However, in general, we were able to show that the following holds: 

\begin{restatable}[]{theorem}{generic}
\label{thm:secondbound}
For every integer $k$ we have
\[ D(n,k) = \frac{n^{2\lfloor k/3\rfloor}}{e^{O_k(\sqrt{\log n})}}\]\end{restatable}

For $k=4,5$, we can also ensure that the lower bound is quadratic. Actually, what we prove is slightly stronger but can be asymptotically summarized with Theorem~\ref{thm:secondbound}. 

The idea of the proof of Theorem~\ref{thm:secondbound} consists in successively adding a graph (inspired by) the construction of Theorem~\ref{thm:3IS} and connecting it in a clever way to the previous graph to increase the diameter quadratically while increasing the size of the independent set by $3$. Note that a super-quadratic lower bound for $k=4$ might lead to an improvement of this general lower bound as long as there is a clever gluing.

Observe that the asymptotic estimate in Theorem~\ref{thm:secondbound} depends on $k$, and hence may not hold when $k$ is not constant, for example when $k$ is linear in $n$. Constructing graphs that maximize the diameter of a connected component in their $k$-configuration graphs (regardless of the value of $k$) is a question raised during the Core Challenge 2022~\cite{SohOI22} for graphs on $10, 50$ and $100$ vertices. Our team proposed a generic construction that obtained the best results. Rewritten in the current formalism, our statement from~\cite{SohOI22} becomes:

\begin{lemma}
\label{lem:core}
For every integer $n$, there exists a graph $G$ on $10n$ vertices such that its $\R_{3n}(G)$ is a path of length $\Theta(4^n)$. In particular $D(n,\frac{3n}{10})=\Omega(2^{n/5})$.
\end{lemma}

Note that we also give a construction showing that $D(n,\frac{2n}{5})=\Omega(2^{n/5})$ (with a slightly worse constant than in Lemma~\ref{lem:core}). Roughly speaking, these graphs are constructed by adding edges between complements of paths on $10$ and $5$ vertices respectively, in a similar fashion to the proof of the upcoming Lemma~\ref{lem:LB_with_2IS}. In particular, those two constructions can be combined and yield the following.

\begin{theorem}
\label{thm:nok}
For every $n$ and every $k$ such that $3n/10\leqslant k \leqslant 2n/5$, $D(n,k)=\Omega(2^{n/5})$. 
\end{theorem}

We believe it is quite surprising that this lower bound holds for such a range of values of $k$, and thus raise the following question.

\begin{question}
What is the asymptotic behavior of $\max_k D(n,k)$?
\end{question}

\section{Generic upper bounds}\label{sec:upper}

We start this section with a preliminary upper bound on $D(n,k)$. 

\begin{lemma}\label{lem:diam-kIS}
$D(n,k)\leqslant {n\choose k-1}$.
\end{lemma}

\begin{proof}
Consider a shortest path $\mathcal{P}$ in the $k$-configuration graph of an $n$-vertex graph $G$. With each edge of $\mathcal{P}$, we associate the $k-1$ vertices of the intersection of the independent corresponding to its endpoints. This defines a mapping from $E(P)$ to sets of $k-1$ vertices of $G$. Since there are $n^{k-1}$ such sets, we simply have to show that this mapping is injective. 
Assume that two distinct edges are mapped to the same set $X$ of $k-1$ vertices. Then $X$ belongs to at least three distinct independent sets that are vertices of $\mathcal{P}$. These three independent sets are pairwise adjacent, which is impossible since $\mathcal{P}$ is a shortest path.
\end{proof}

We will see that this bound is sharp for $k=2$. However, when $k$ increases this bound can be slightly improved, as summarized in Theorem~\ref{thm:bornesup_gen} that we recall below.

\littleo*

\begin{proof}
Consider a graph $G$ on $n$ vertices whose configuration graph has maximum diameter $d$. 
Let $\P=Z_1, Z_2, \ldots, Z_d$ be a shortest path of length $d$ in $\R_k(G)$. Let us partition the nodes in $\P$ into two sets $\P_1$ and $\P_2$ where $\P_1$ (resp. $\P_2$) is the set of odd (resp. even) nodes of $\P$. Note that if we consider two subsets of $\P_i$ for $i \le 2$ then their intersection has size at most $k-2$ (otherwise $\P$ would not be induced). 

For every $i \le 2$, let $H_i$ be the $(k-1)$-uniform hypergraph whose vertices are the same as for $G$ and whose hyperedges are the independent sets of size $k-1$ contained in some set of $\P_i$. Moreover, denote by $K$ the $(k-1)$-uniform hyperclique on $k$ vertices. Observe that by construction, each $Z\in \P_i$ creates (exactly) one copy of $K$ in $H_i$. Also note that every subset of size $k-1$ of such a $Z$ belongs to exactly one independent set of $\P_i$ since otherwise the two independent sets would be adjacent, contradicting the minimality of $\P$.  We now distinguish two cases:
\smallskip

\noindent
\textbf{Case 1.} $H_i$ contains more than $n^{k-1}$ copies of $K$. \\ Since by Lemma~\ref{lem:diam-kIS}, at most $n^{k-1}$ are created by some $Z\in\P_i$, there exists a copy of $K$ in $H_i$ such that $V(K)\notin \P_i$. Consider now three hyperedges $e_1,e_2,e_3$ in $K$ that pairwise intersect on $k-2$ vertices (note that this is possible since $k\geqslant 3$). 
    
    By construction, each of these hyperedges are contained in some element of $\P_i$, so there exist $x_1,x_2,x_3\in V(G)$ such that $e_j\cup\{x_j\}\in \P_i$ for $j=1,2,3$. In particular, each $e_j$ is an independent set of $G$, therefore $e_1\cup e_2\cup e_3$ is also an independent set of $G$ of size $k$. Therefore all the $e_j\cup\{x_j\}$'s are at distance at most 2 from each other in $\R_k(G)$ since all of them are adjacent to $e_1\cup e_2\cup e_3$. This is a contradiction since $\P$ is a shortest path and $\P_1$ (resp. $\P_2$) only contains even (resp. odd) vertices of $\P$ and then two of the three independent sets $e_j\cup\{x_j\}\in \P_i$ for $j \le 3$ should be at distance at least $4$.
    \smallskip

\noindent \textbf{Case 2.} $H_i$ contains at most $n^{k-1}=o(n^k)$ copies of $K$. \\
By the hypergraph removal lemma~\cite{rodl2005hypergraph,gowers2007hypergraph}, there exists a set $S$ of hyperedges of $H$ such that $|S|=o(n^{k-1})$ and $H-S$ contains no copy of $K$. Recall that each hyperedge of $S$ is contained in exactly one element of $\P_1$, and each element of $\P_i$ creates a copy of $K$ in $H$, therefore we get $|\P_i|\leqslant |S|=o(n^{k-1})$.

To conclude, observe that $d\leqslant 2|\P_i|=o(n^{k-1})$ for every $i \le 2$.
\end{proof}

\section{Lower bounds}\label{sec:lower}

\subsection{Independent sets of size 2}

In this section, our main goal is to prove Theorem~\ref{thm:2IS} that we recall below.

\twoIS* 

We thus consider the independent sets of size $2$ of a graph $G$. Note that these sets are exactly the non-edges of $G$, \emph{i.e.} the edges of $\bar{G}= (V, \P_2(V) \setminus E)$. Therefore, we get the following observation. 

\begin{observation}
The configuration graph $\R_2(G)$ is the line graph of $\bar{G}$. 
\end{observation}

Note that for every graph $G$, any induced path on $p$ vertices in $L(G)$ corresponds to a path on $p$ edges in $G$. In particular, we derive two consequences.

\begin{observation}
Let $A,B$ be two independent sets of $G$ of size $2$ and $a \in A, b \in B$. There is a TJ-transformation from $A$ to $B$ if and only if $a,b$ are in the same connected component of~$\overline{G}$.
\end{observation}

We can also obtain the following which ensures that the bound of Lemma~\ref{lem:diam-kIS} is tight:

\begin{lemma}
For every $n$-vertex graph $G$, \[\diam(\R_2(G))=\diam(L(\bar{G}))\leqslant \diam(\bar{G})-1\leqslant n-2.\]
\end{lemma}

Note that the last bound is tight only when $\bar{G}$ is a path, which concludes the proof of Theorem~\ref{thm:2IS}.

Since the diameter is linear, one might wonder if we can determine in linear time if there exists such a transformation (and find it). Note that we cannot just compute the line graph of the complement and run a BFS on it. Indeed, even if a BFS can be computed in linear time with respect to the number of edges of its input, this number may be quadratic with respect to the number of edges of the original graph. However, by complementing only the graph induced by vertices of large degree, we obtain the following.

\begin{theorem}
\label{thm:2IS_linear}
Let $A,B$ be two independent sets of $G$ of size $2$. We can decide if there exists a TJ-transformation from $A$ to $B$ in time $O(|V(G)|+|E(G)|)$.
\end{theorem}

\begin{proof}
Let $G$ be an $n$-vertex $m$-edge graph, and $s,t$ two vertices of $G$. We start by precomputing the degrees of the vertices of $G$ in $O(n+m)$ time. Let $B$ be the set of vertices of degree at least $\frac{n-1}{2}$, and $S=V(G)\setminus B$. Observe that by the pigeonhole principle, any two vertices in $S$ must have a common non-neighbor, hence are connected in $\bar{G}$. 
Let us denote by $H$ the graph obtained by identifying all the vertices of $S$ into a single vertex $x$ and where we put an edge between $x$ and $y \notin S$ if $y$ is adjacent to a vertex of $S$. It is easy to check that there is a path between $s$ and $t$ in $\overline{G}$ if and only if there is such a path in $\overline{H}$ (up to replacing $s$ or $t$ by $x$ if they lie in $S$).  Note moreover that one can easily compute the graph $H$ in $O(n+m)$ time. 

One can notice that the graph $H$ might be sparse and then its complement can have size $\Omega(|V(H)|^2)$. However, observe that 
\[(|V(H)|-1)\times \frac{n-1}{2}\leqslant \sum_{v\in B} \deg_G(v) \leqslant 2m,\]
hence $|V(H)|=O(\frac{m}{n})$. In particular, one can compute $\bar{H}$ and use a BFS in $\bar{H}$ in time $O(|V(H)|^2)=O(\frac{m^2}{n^2})=O(m)$.
\end{proof}

Note that the algorithm we provide can easily be adapted to return a (possibly non-optimal) transformation when it exists.

\subsection{Almost quadratic construction for independent sets of size 3}\label{sec:3BI}

The rest of this section is devoted to prove the following result:

\threeIS*

The proof is based on two steps. First, we prove that there exists a graph whose configuration graph is the disjoint union of $n/e^{O(\sqrt{\log n})}$ paths of linear length. We then prove that, starting from a graph whose configuration graph is disconnected, we can (up to adding few vertices), obtain a graph whose configuration graph is connected and whose diameter is at least the sum of the diameter of the connected components of the initial configuration graph. 
While the first step is specific to $k=3$ and is based on the existence of almost linear subsets of integers without arithmetic sequences of length $3$, the gluing process is general and holds for any possible value of $k$. 
Let us first prove the gluing lemma.

\begin{lemma}\label{lem:joincomp}
Let $k \ge 3$.
Let $G$ be a graph on $n$ vertices whose $k$-configuration graph contains $r$ connected components $\C_1,\ldots,\C_r$ of diameter respectively $d_1,\ldots,d_r$. Then there exists a graph on at most $n+(3k-2) \cdot (r-1)$ vertices whose configuration graph has diameter at least $(4k-4)(r-1)+\sum_{i \le r} d_i$.
\end{lemma}
\begin{proof}

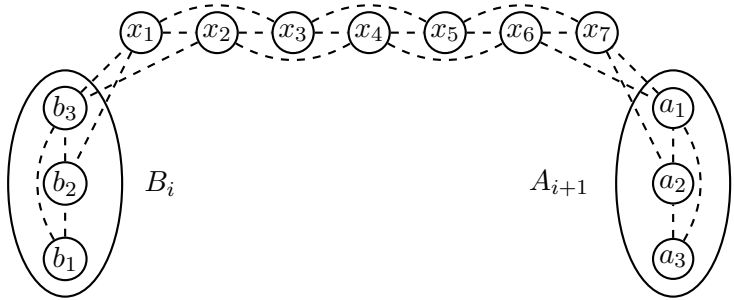
\begin{figure}[!ht]
\centering
\begin{tikzpicture}[thick,v/.style={draw,circle,inner sep =1pt}]
\node[v] (x1) at (0,0) {$x_1$};
\node[v] (x2) at (1,0) {$x_2$};
\node[v] (x3) at (2,0) {$x_3$};
\node[v] (x4) at (3,0) {$x_4$};
\node[v] (x5) at (4,0) {$x_5$};
\node[v] (x6) at (5,0) {$x_6$};
\node[v] (x7) at (6,0) {$x_7$};
\draw[dashed] (x1) -- (x2) -- (x3) -- (x4) -- (x5) -- (x6) -- (x7);
\draw[bend left,dashed] (x1) to (x3);
\draw[bend left,dashed] (x3) to (x5);
\draw[bend left,dashed] (x5) to (x7);
\draw[bend right,dashed] (x2) to (x4);
\draw[bend right,dashed] (x4) to (x6);

\node[v] (b1) at (-1,-3) {$b_1$};
\node[v] (b2) at (-1,-2) {$b_2$};
\node[v] (b3) at (-1,-1) {$b_3$};
\draw[dashed] (b1) -- (b2) -- (b3);
\draw[bend left,dashed] (b1) to (b3);
\draw[dashed] (x2) -- (b3) -- (x1) -- (b2);
\draw (b2) ellipse (.75cm and 1.5cm);
\node at (.25,-2) {$B_i$};

\node[v] (b1) at (7,-3) {$a_3$};
\node[v] (b2) at (7,-2) {$a_2$};
\node[v] (b3) at (7,-1) {$a_1$};
\draw[dashed] (b1) -- (b2) -- (b3);
\draw[bend right,dashed] (b1) to (b3);
\draw[dashed] (x6) -- (b3) -- (x7) -- (b2);
\draw (b2) ellipse (.75cm and 1.5cm);
\node at (5.5,-2) {$A_{i+1}$};
\end{tikzpicture}
\caption{Construction for Lemma~\ref{lem:joincomp} with $k=3$, where only non-edges are drawn.}
\label{fig:joincomp}
\end{figure}

For every $i \le r$, let us denote by $A_i$ and $B_i$ two independent sets at distance $d_i$ in the component $\C_i$ of the $k$-configuration graph. For every $i$, we denote by $a_j^i$ (resp. $b_j^i$) the $j$-th vertex of $A_i$ (resp. $B_i$). We moreover assume that $a_k^i$ and $b_1^i$ are respectively the first and last vertices modified in a shortest sequence $P_i$ from $A_i$ to $B_i$. Note that the sets $(A_1,\ldots,A_r,B_1,\ldots,B_r)$ might intersect.

For every $i \le r-1$, we create $3k-2$ new vertices $x_1^i,\ldots,x_{3k-2}^i$ in order to connect $B_i$ to $A_{i+1}$ in the configuration graph of the new graph (see Figure~\ref{fig:joincomp} for an illustration of the construction). We first add all the edges between the new vertices and $V(G)$ and, for every $i \ne j$, the vertices $x_a^i$ and $x_b^j$ are adjacent regardless of $a,b$. 
Moreover, for every $p < q \le 3k-2$, $x_p^i$ is adjacent to $x_q^i$ if and only if $q-p > k-1$.
We finally remove the following edges: for every $j$ we remove the edges between $x_j^i$ (resp. $x_{3k-2-j}$) and $b_{j'}^i$ (resp. $a_{k-j'}^{i+1}$) with $j' > j$. Let us denote by $H$ the resulting graph.

Let $X_i$ be the (ordered) set of vertices $\{b_1^i,\ldots,b_k^i,x_1^i,\ldots,x_{3k-2}^i,a_1^{i+1},\ldots,a_k^{i+1}\}$.
Let us first prove the following simple claim on the structure of independent sets:
\begin{claim}\label{clm:inter}
Let $i \le r-2$. Every $k$-independent set $S$ containing one vertex in $\{x_1^i,\ldots,x_{3k-2}^i \}$:
\begin{itemize}
    \item consists of $k$ consecutive vertices of $X_i$ and,
    \item has degree $2$ in $\R_k(H)$ if it contains a vertex in $\{x_2^i,\ldots,x_{3k-3}^i\}$.
\end{itemize}
\end{claim}
\begin{proof}
Let us first prove that if an independent set $S$ contains a vertex in $\{ x_1^i,\ldots,x_{3k-2}^i \}$ then it contains consecutive vertices of $X_i$.

If $S\subseteq \{x_1^i,\ldots,x_{3k-2}^i\}$, then let us denote by $x_a^i$ and $x_b^i$ the first and last vertices in $S$. By construction, since $S$ is an independent set, we must have $b-a \le k-1$. And then $S$ contains only non-neighbors of $x_a^i$ and $x_b^i$, hence $b-a=k-1$ and $S=\{x_a^i,\ldots,x_b^i\}$. So from now on we can assume that $S$ contains a vertex of $V(G)$.

Since $x_{k-1}^i,\ldots,x_{2k-2}^i$ are complete to $G$, the set $S$ cannot contain one of these vertices. And since $\{x_1^i,\ldots,x_{k-2}^i\}$ is complete to $\{x_{2k-1}^i,\ldots,x_{3k-2}^i\}$, we can assume by symmetry that $S$ contains vertices in $\{x_1^i,\ldots,x_{k-2}^i\}$ but not in $\{x_{2k-1}^i,\ldots,x_{3k-2}^i\}$. Let us denote by $a\leqslant k-1$ the largest index such that $x_a^i\in S$. By construction, $x_a^i$ is non-adjacent to the $k-1$ vertices before it in the sequence and complete to all the other vertices of $H \setminus X_i$. So $S=\{b_{a+1}^i,\ldots,b_k^i,x_1^i,\ldots,x_a^i\}$. 

For the second item, observe that indeed, each set of $k$ consecutive vertices in $X_i$ is independent, and is connected to the independent sets corresponding to the $k$ vertices just before and after them in the ordering. Moreover, each independent set $S$ intersecting $\{x_2^i,\ldots,x_{3k-3}^i\}$ contains at least two vertices in $\{x_1^i,\ldots,x_{3k-2}^i\}$. Therefore $S$ is only adjacent to independent sets containing at least one vertex in $\{x_1^i,\ldots,x_{3k-2}^i\}$. By the first item, they consist of $k$ consecutive vertices of $X_i$, hence $S$ has exactly two neighbors in $\R_k(H)$. 
\end{proof}

So the $k$-configuration graph of $H$ restricted to $X_i$ induces a path $\P_i$ from $B_i$ to $A_{i+1}$ of length $4k-2$.  
By concatenating these paths with shortest reconfiguration sequences from the $A_i$ to $B_i$ for every $i \le r$, we get a reconfiguration sequence $\P$ from $A_1$ to $B_r$ of length $(4k-2)(r-1)+\sum_{i \le r} d_i$.

To complete the proof we have to prove that we can shorten the sequence $\P$ by exactly $2(r-1)$ steps (and that no shorter transformation exists). Indeed, for every $i \le r-1$, consider a reconfiguration sequence from $A_i$ to $B_i$ of minimum size where $a_1^i$ is the vertex deleted from $A_i$ at the beginning of the sequence and $b_1^i$ is the last vertex to be moved on $B_i$ at the end of the sequence (this reconfiguration sequence exists by assumption). If we denote by $B_i'$ the independent set before $B_i$, $B_i'$ contains $B_i \setminus \{ b_1^i \}$. Then $B_i'$ is adjacent to $(B_i \cup x_1^i) \setminus b_1^i$ in the configuration graph and then we can remove $B_i$ in the reconfiguration sequence from $A_1$ to $B_r$ and still have a reconfiguration sequence. Similarly, we can find a shortcut of the sequence on $A_i$ for every $2 \le i \le r$. So we can shorten $\P$ by $2(r-1)$ steps.

We claim that this transformation has shortest length. Let us briefly argue why it is true. Consider an independent set of $H$ containing exactly one vertex in $\{x_j^i\mid i \le r-1, j \le 3k-2\}$. It should contain the vertex $x_1^i$ or $x_{3k-2}^i$ for some $i$ by Claim~\ref{clm:inter} and $k-1$ vertices of an independent set in $\C_i$ or in $\C_{i+1}$. Thus it can only be adjacent to an independent set of the component of $\C_i$ or $\C_{i+1}$ or an independent set containing two vertices of $X_i$ by Claim~\ref{clm:inter}. 

Using Claim~\ref{clm:inter} again, it means that from an independent set only containing $x_1^i$ (resp. $x_{3k-3}^i$) we can only reach an independent set of $\C_{i+1}$ (resp. $\C_i$) containing $B_{i-1} \setminus b_1^{i-1}$ (resp. $A_{i+1} \setminus a_k^{i+1}$).\end{proof}

Before proving that it is possible to obtain an almost linear number of components of almost linear size, we need some definitions and results of group theory.

Let $S$ be a set of integers. We say that $S$ is \emph{$3$-AP-free} if it does not contain an arithmetic progression of length $3$, i.e. there does not exist $s_1<s_2<s_3$ in $S$ such that $s_2-s_1 = s_3-s_2$. Determining the size of the largest possible $3$-AP-free subset of $[1,n]$ is a heavily studied problem whose exact answer is not known. It was shown that there does not exist any $3$-AP-free set of positive density in $\mathbb{N}$ \cite{Klaus52}. However Behnrend proved in~\cite{behrend1946sets} that there exist $3$-AP-free subsets of $\{1,\ldots,n \}$ of size $n/e^{O(\sqrt{\log n})}$. Note that if $S$ is $3$-AP-free, then the set $4S+1$ also is $3$-AP-free.  So, there exists a $3$-AP-free sequence of size $n/e^{O(\sqrt{\log n})}$ only containing integers whose value is $1$ modulo $4$. Such a set will be called an \emph{odd $3$-AP-free} sequence in the rest of the paper. 

Now let $p$ be a prime number. It is well known that, for every $\alpha \in \{1,\ldots,p-1\}$, the sequence of the $k\alpha$ modulo $p$ is a periodic sequence of period $p$. That is $\{0,\alpha,\ldots,(p-1)\alpha\} = \{0, \ldots, p-1\}$ modulo $p$.

We claim that the following holds:

\begin{lemma}\label{lem:manycomp}
Let $n$ be a prime number. Let $S$ be an odd $3$-AP-free sequence where the maximum integer is at most $n/8$. Then, there exists a graph on $n-1$ vertices whose configuration graph is the disjoint union of $|S|$ paths of length $n-3$.
\end{lemma}
\begin{proof}
Let $G$ be the graph obtained from a clique on $n$ vertices by removing the edges $(i,i+s)$ and $(i,i+2s)$ for $s \in S$ and $i \le n$, where all integers are understood modulo $n$. 

We claim that $\R_3(G)$ consists of $|S|$ induced cycles of length $n$. First observe that for each $s\in S$ we have $n$ independent sets, namely $\{ks,(k+1)s,(k+2)s\}$ ($0\leqslant k <n$), which induce a cycle. In particular, $\R_3(G)$ contains $|S|$ cycles of length $n$.

Let us now show that this union of cycles is induced. To this end, let $I=\{a,b,c\}$ and $I'=\{a',b',c'\}$ be two independent sets such that $c-b = b-a = s \in S$ and $c'-b'=b'-a'=s' \in S$ with $s\neq s'$. If $I\cap I'$ contains two elements $x$ and $y$, then observe that $x-y\in\{\pm s,\pm 2s\}\cap\{\pm s',\pm 2s'\}$, which is impossible since $s$ and $s'$ are distinct integers below $n/8$ and are both $1$ mod $4$. Therefore $I$ and $I'$ are not adjacent.

Finally, we prove that $\R_3(G)$ has no other vertex. Let $\{a, b, c\}$ be an independent set of $G$. If $b-a \in S$ and $c-b \in S$, then $c-a$ is even and lies in $2S$, hence we can write $b-a=s, c-b=s'$ and $c-a=2s''=s'-s$.
Since $S$ is $3$-AP-free, we have $s=s'=s''$ and then $b-a = c-b$ (and then $\{a,b,c\}$ is one of the sets described above). 
Otherwise, $b-a$ or $c-b$ does not belong to $S$, thus $c-a$ is equal to $0$ or $3$ modulo $4$ and then $c-a \notin S \cup 2S$, which is a contradiction. 

So $\R_3(G)$ is the disjoint union of cycles with no edges between them. Now if we remove the vertex $0$ from $G$, each of these cycles now becomes a path (we remove the three independent sets containing it for every $S$), which completes the proof.
\end{proof}

Let us combine Lemma~\ref{lem:joincomp} and~\ref{lem:manycomp} to prove Theorem~\ref{thm:3IS}.
Let $S$ be an odd $3$-AP-free sequence of size $n/e^{O(\sqrt{\log n})}$. By Lemma~\ref{lem:joincomp}, there exists a graph $G$ on $n-1$ vertices whose $3$-configuration graph admits $|S|$ connected components of diameter $n-3$. By applying Lemma~\ref{lem:manycomp} to $G$, we obtain a new graph on $n - 1 + 7 \cdot (|S| - 1) \le 8n$ vertices and whose $3$-configuration graph has diameter at least $8\cdot (|S|-1) + |S|\cdot (n-3) = n^2/e^{O(\sqrt{\log n})}$, which completes the proof.

We end this section with the following question which would extend Theorem~\ref{thm:2IS_linear} to independent sets of size $3$.

\begin{question}
Can we compute the diameter or the existence of a transformation between two independent sets of size $3$ in (sub)quadratic time?
\end{question}

\subsection{General lower bound}

The goal of this section is to generalize the construction of Section~\ref{sec:3BI} to larger values of $k$. Unfortunately, we were not able to obtain a lower bound that almost fits the upper bound for larger values of $k$, but we still obtain the following. 

\generic*

Let us first give the flavour of the proof with an intermediate construction.

\begin{lemma}\label{lem:LB_with_2IS}
Let $G$ be a graph with independence number $k$ such that $\R_k(G)$ has diameter $d$. Then for every integer $n$, there exists a graph $H$ on $|V(G)|+6n+2$ vertices such that $\R_{k+2}(H)$ has diameter at least $2dn$.
\end{lemma}
\begin{proof}
Let $A$ and $B$ be two independent sets of size $k$ at distance $d$ in $\R_k(G)$. 

\begin{figure}[!ht]
    \centering
    \includegraphics[scale=0.65]{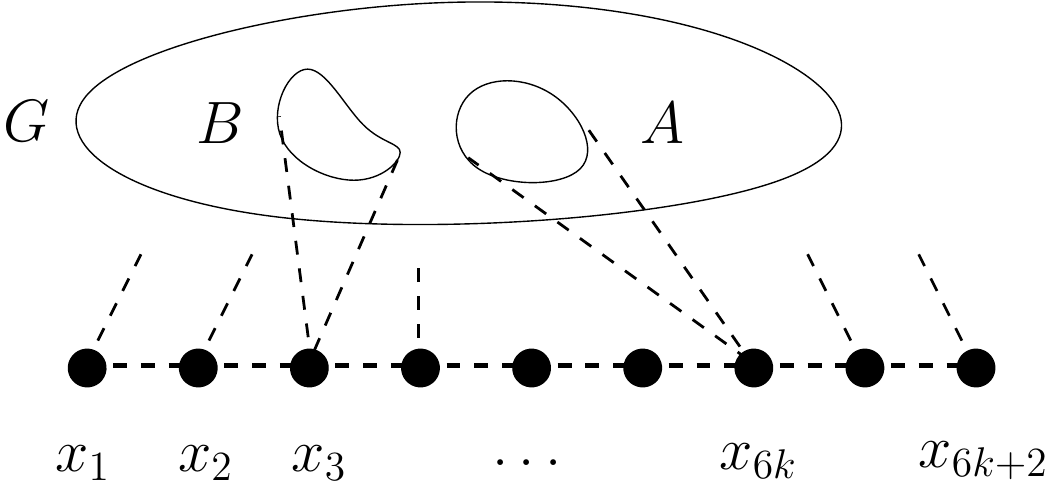}
    \caption{Construction of the proof of Lemma~\ref{lem:LB_with_2IS}. For readability, we have only represented the non-edges incident with the vertices of $X$.}
    \label{fig:BIvia2IS}
\end{figure}

Let us create $6n+2$ vertices $X=\{ x_1, x_2, \ldots, x_{6n+2} \}$, inducing the complement of a path. For every $1 \le \ell \le n$, link the vertices $x_{6\ell-3}$ and $x_{6\ell}$ to respectively $V(G) \setminus B$ and $V(G)\setminus A$. Let us denote by $H$ the resulting graph (see Figure~\ref{fig:BIvia2IS} for an illustration). 
Let $C, D$ be the independent sets $A \cup \{ x_1,x_2\}$ and $A \cup \{x_{6n+1},x_{6n+2} \}$. Since $k$ is the independence number of $G$ and $X$ is the complement of a path, the maximum independent sets of $H$ have size $k+2$ and all of them contain $k$ vertices in $V(G)$ and two vertices in $X$.

In particular, in every transformation from $C$ to $D$, two tokens are present at each time in $X$, and the movements of these tokens correspond to a path in $\R_2(X)$. Therefore, in order to slide the tokens from $\{x_1,x_2\}$ to $\{x_{6n+1},x_{6n+2}\}$, each transformation must hit in order some independent sets $S_1,\ldots,S_{6n+1}$ such that $S_i\cap X=\{x_i,x_{i+1}\}$. Amongst all such independent sets, we select $S_i$ as the first such independent set. Note that for every $i = 3 \mod 6$ (resp. $0 \mod 6$), $S_i= B \cup \{x_i,x_{i+1} \}$ (resp. $S_i= A \cup \{x_i,x_{i+1} \} $). So, for every $i =0$ or $3 \mod 6$, in order to transform $S_i$ into $S_{i+3}$ we need at least $d+3$ steps (since we need to transform $A$ into $B$ plus three token slides on $X$).

Therefore, the length of the reconfiguration sequence from $C$ to $D$ is at least $2n(d+3)$, which completes the proof.
\end{proof}

Since the graph $H$ constructed in Lemma~\ref{lem:LB_with_2IS} has maximum independent sets of size $k+2$, we can iterate the process starting from the complement of a path and prove that the following holds:

\begin{corollary}\label{coro:k/2}
For any $k$, $D(n,k)=\Omega_k(n^{\lfloor k/2 \rfloor})$.
\end{corollary}

In the proof of Lemma~\ref{lem:LB_with_2IS}, we can see the vertices of indices $0$ modulo $3$ as toll booths which enforces us to perform a lot of modifications in $G$ in order to pass though these vertices.
To improve the bound of Corollary~\ref{coro:k/2}, we will generalize Lemma~\ref{lem:LB_with_2IS}. Indeed instead of gluing the complement of a path to $G$ and increasing the size of the independent sets by $2$, we will copy a (slightly modified) copy of the graph of Theorem~\ref{thm:3IS} and increase the size of the independent sets by $3$. Note that we cannot automatically, when we have a graph with a large reconfiguration diameter, glue it with another graph and get a large reconfiguration diameter since we need to be careful on where we put the toll booths. (That is why the graph of Theorem~\ref{thm:3IS} has to be slightly modified to work in our setting.) The main ingredient for Theorem~\ref{thm:secondbound} is thus the following analogue of Lemma~\ref{lem:LB_with_2IS}.

\begin{lemma}\label{lem:LB_with_3IS}
Let $n,k$ be two integers.
Let $G$ be a graph with maximum independent sets of size $k$ such that the $k$-configuration graph of $G$ has diameter $d$. Then there exists a graph on at most $|V(G)|+3kn$ vertices whose $(k+3)$-reconfiguration diameter is at least $\frac{dn^2}{e^{O(\sqrt{\log n})}}$.
\end{lemma}
\begin{proof}
The construction is inspired from Section~\ref{sec:3BI}. Let $S'$ be a largest $3$-AP-free sequence in $[1,n/64]$ and $S$ be the set $8S'+1$. Recall that $|S|= \frac{n}{e^{O(\sqrt{\log n})}}$. Denote by $H$ the graph obtained applying Lemma~\ref{lem:manycomp} to $S$. Recall that vertices of $H$ can be labeled from $1$ to $n-1$ in such a way that:
\begin{enumerate}
    \item the independent sets of size 3 have consecutive values modulo 8, and
    \item if two such sets are adjacent in $\R_3(H)$, then their symmetric difference contains two elements whose difference is $\pm 3$ mod 8.
    \item each vertex of $H$ appears in at least one independent set in each connected component of $\R_3(H)$.
\end{enumerate}

Let $A$ and $B$ two independent sets of $G$ at distance $d$ from each other in $\R_k(G)$. Denote by $G'$ the graph obtained by taking a copy of $G$ and $H$ and adding all the edges:
\begin{itemize}
    \item between $V(G)\setminus A$ and vertices of $H$ that are 0 mod 8.
    \item between $V(G)\setminus B$ and vertices of $H$ that are 4 mod 8.
\end{itemize}

Note that since $G$ has no independent set of size more than $k$, then any independent set of $G'$ of size $k+3$ decomposes as $k$ vertices of $G$ and $3$ vertices of $H$. In particular, any reconfiguration sequence from $I$ to $J$ in $G'$ yields a reconfiguration sequence from $I\cap V(H)$ to $J\cap V(H)$ in $H$ (and the same holds for $G$). Therefore $\R_{k+3}(G')$ contains at least as many connected components as $\R_3(H)$. 

Let $X_1,\ldots, X_r$ be a connected component of $\R_3(H)$ which induces a path.  By construction, we may assume that $X_1$ contains vertices that are 1,2, and 3 mod 8, and by (1) and (2), each $X_i$ contains vertices that are $i,i+1$ and $i+2$ mod 8. Up to reducing $r$ by at most $7$, we may even assume that $X_r$ also contains vertices that are $1,2$ and $3$ mod $8$. Thus $X_1\cup A$ and $X_r\cup A$ are independent sets of $G'$.

Observe that if $i=1\mod 4$, there is no edge between $X_i$ and $V(G)$, hence there is a reconfiguration sequence of length $d$ between $X_i\cup A$ and $X_i\cup B$. Therefore, one can reach $X_r\cup A$ from $X_1\cup A$ going through the following steps: $X_1\cup B,X_2\cup B,X_3\cup B,X_4\cup B, X_5\cup B, X_5\cup A,\ldots$.

Therefore, $X_1\cup A$ and $X_r\cup A$ are in the same connected component of $\R_{k+3}(G')$. Let us now compute (a lower bound on) their distance. Consider a shortest reconfiguration sequence between $X_1\cup A$ and $X_r\cup A$. Recall that this yields a (non-necessarily shortest) reconfiguration sequence from $X_1$ to $X_r$ in $H$. By (3), for every vertex of $H$ which is $0$ mod $8$, we may choose a set $X_i$ that contains it, and denote by $X_{i_1},\ldots, X_{i_{n/8}}$ the subsequence they form (observe that this is well-defined since no $X_i$ can contain two vertices that are 0 mod 8 by (1)). 

Now by (1) and (2), note that in the reconfiguration sequence between every $X_{i_j}$ and $X_{i_{j+1}}$, there must exist some independent set $X_{i'_j}$ containing a vertex that is $4$ mod $8$. By construction, the only independent set of $G'$ containing each $X_{i_j}$ (resp. $X_{i'_j}$) is $X_{i_j}\cup A$ (resp. $X_{i'_j}\cup B$). Therefore the distance between $X_{i_j}\cup A$ and $X_{i'_j}\cup B$ is at least $d$, and so is the distance between $X_{i'_j}\cup B$ and $X_{i_{j+1}}\cup A$. 

Therefore, the distance between $X_1\cup A$ and $X_r\cup A$ is at least $2d\times (\frac{n}{8}-1)=\frac{dn}{4}-2d$. Since $\R_{k+3}(H)$ contains at least $|S|$ components of diameter at least $\frac{dn}{4}-2d$, applying Lemma~\ref{lem:joincomp} to $G'$ yields a graph on $|V(G')|+(3k-2)(|S|-1)= |V(G)|+n-1+(3k-2)(|S|-1)$ vertices with diameter at least $(4k-4)(|S|-1) + |S|dn/4-2d|S|$, which concludes since $|S|=n/e^{O(\sqrt{\log n})}$.

\end{proof}

As an immediate corollary, we obtain Theorem~\ref{thm:secondbound}.

\section{Conclusion}

Let us start the conclusion with this simple remark:

\begin{lemma}\label{lem:3IS-path}
Let $G$ be a graph. There exists a super-graph $G'$ of $G$ with the same number of vertices such that $\R_3(G')$ is a path and the largest diameters of the connected components of $\R_3(G)$ and $\R_3(G')$ are the same.
\end{lemma}

\begin{proof}
We start by adding arbitrarily edges to $G$ while the largest diameter of a component in $\R_3(G)$ stays unchanged. To conclude, we show that $\R_3(G)$ is a path. Let $\P$ be a shortest path of maximal length in $\R_3(G)$.

Assume that there is a node $Z=\{u,v,w\}$ of $\R_3(G)$ that is not in $\P$. For each $x\neq y\in Z$, adding the edge $xy$ to $G$ decreases the diameter of the component of $\P$ in $\R_3(G)$. Hence there must be an independent set of $\P$ that contains both $x$ and $y$.

Therefore, we can assume that $\P$ contains three independent sets $\{u,v,a\},\{u,w,b\}$ and $\{v,w,c\}$, which are all pairwise distinct since $Z \notin \P$. Since $\P$ is a shortest path and these sets are all neighbors of $Z$, they must be consecutive in $\P$. Therefore, we have $a=b=c$. But then, $\{u,v,a\}, \{v,w,a\}$ and $\{u,w,a\}$ induces a triangle in $\R_3(G)$, a contradiction since $\P$ is induced.
\end{proof}

Note that all the graphs obtained from our constructions also satisfy that their configuration graphs are paths. We conjecture that the following is true in general:

\begin{conjecture}
For every $k \ge 2$ and every $n$, there exists a graph $G$ on $n$ vertices maximizing $D(n,k)$ and such that the $\mathcal{R}_k(G)$ is a path.
\end{conjecture}

Note that this result holds for $k=2$ (since complement of paths are tight) and for $k=3$ as proven in Lemma~\ref{lem:3IS-path}. 

Another interesting question is the following. We have remarked that both lower and upper bounds of the $(6,3)$-problem correspond to the bounds obtained for the largest possible diameter of $\R_3(G)$. Are the two problems equivalent (up to a multiplicative constant)? While the existence of a better diameter for some graph $G$ would immediately imply a better bound for the $(6,3)$-problem (by simply considering even vertices of a shortest path), the converse is not immediate.

\paragraph{Acknowledgments.}
This work started thanks to the \href{https://core-challenge.github.io/2022/}{CoRe programming challenge}~\cite{SohOI22} whose topic in 2022 consisted in finding the graphs on $10,50$ and $100$ vertices with the largest possible independent set reconfiguration diameter.

\bibliographystyle{abbrv}

\end{document}